\documentclass[12pt]{amsart}
\usepackage{fullpage}
\usepackage[utf8]{inputenc}
\usepackage[english]{babel}
\usepackage{amsthm}
\usepackage{amsmath}
\usepackage{amssymb}
\usepackage{multirow}
\usepackage{hyperref}

\newtheorem{theorem}{Theorem}[section]

\newtheorem{corollary}{Corollary}[theorem]
\newtheorem{lemma}[theorem]{Lemma}
\newtheorem{proposition}[theorem]{Proposition}
\newtheorem*{remark}{Remark}

\newcommand{\QQ}{\mathbb{Q}}
\newcommand{\RR}{\mathbb{R}}
\newcommand{\ZZ}{\mathbb{Z}}
\newcommand{\CC}{\mathbb{C}}

\DeclareMathOperator{\Nm}{Nm}

\DeclareMathOperator{\Norm}{Norm}

\DeclareMathOperator{\Tr}{Tr}

\DeclareMathOperator{\Hom}{Hom}

\DeclareMathOperator{\GL}{GL}

\DeclareMathOperator{\SU}{SU}
\DeclareMathOperator{\Sp}{Sp}

\DeclareMathOperator{\MT}{MT}

\begin{document}

\title{Constructing a CM Mumford fourfold from Shioda's fourfold}
\author{Yuwei Zhu}
\maketitle

\begin{abstract}

In \cite{Shi81} Shioda proved that the Jacobian $A_S$ of the curve $y^2 = x^9 -1$ is a 4-dimensional CM abelian variety with codimension 2 Hodge cycles not generated by divisors. It was noted by Shioda that this behavior resembles the abelian varieties constructed by Mumford in \cite{Mum69}. We prove that Shioda's fourfold $A_S$ cannot be realized as a special case of Mumford's construction. However, by modifying its Hodge structure, we construct a basis for computing the period matrix of a CM Mumford fourfold with multiplication by $\sqrt{-3}$.

\end{abstract}

\section{Introduction}

In \cite{Mum69} Mumford constructed families of abelian fourfolds with exceptional Hodge cycles in their self-products (see \cite{Moo04} Example 5.9). Although the families constructed by Mumford are one-dimensional, whether they intersect with the moduli space of Jacobians of genus 4 curves is unknown (Problem 1 in \cite{EMO00}). It was noted by Shioda in \cite{Shi81} (Section 2) that, up to isogeny, the Jacobian $A_S$ of the genus 4 curve $y^2 = x^9 -1$ demonstrates similar behavior as Mumford's construction, which naturally leads to the question as to whether it can be realized as a special point of Mumford's loci. The story of CM Mumford fourfolds remains mysterious, with some results proven over local places by Noot in \cite{Noo00}.

In this paper we will introduce the Mumford-Tate group of CM abelian varieties, some basic facts about Mumford-Tate groups regarding products of abelian varieties and dual abelian varieties, and calculate the corresponding representations of CM Mumford fourfold (denoted $A_M$) and $A_S$ respectively. By showing that $A_S$ and general $A_M$ does not share the same CM type, we prove the following:

\begin{theorem}
 $A_S$ itself is not a special case of Mumford's construction (i.e. it is not parameterized by any Shimura curve that parameterizes Mumford constructions). With respect to the canonical basis given by Shioda, a CM Mumford fourfold $A_M$ can be constructed by twisting the polarization of $A_S$ by a fundamental unit in $\ZZ(\zeta_9 + \overline{\zeta_9})$.
\end{theorem}

\renewcommand{\abstractname}{Acknowledgments}

\begin{abstract}

The author would like to thank Prof. Kiran Kedlaya for pointing to the subject of Mumford's construction; Prof. Andrew Sutherland for sharing his knowledge on the Mumford-Tate representation of Shioda's fourfold; Prof. John Voight, Prof. Everette Howe and Prof. Christophe Ritzenthaler for the valuable feedback of a error in the previous version of this paper; Prof. Yuri Zarhin for pointing to the language of Weil type abelian varieties; and Prof. Brendan Hassett for his valuable suggestions and kind support. The author is supported by NSF grants DMS-1551514 and DMS-1701659.

\end{abstract}
\section{General facts about Mumford-Tate groups of CM abelian varieties}

\subsection{Mumford-Tate representations}

The main reference for this section is \cite{Del82} chapter 3. 

We begin with the classical story: a Hodge structure on a rational vector space $V$ of weight $1$ is a homomorphism of real algebraic groups:
$$h: \CC^{\times} \rightarrow GL(V\otimes \RR)$$
such that $V \otimes \CC$ admits a decomposition into $ V^{1,0} \oplus V^{0,1}$ satisfying:

\begin{enumerate}
 \item $V^{1,0}$ is the complex conjugate of $V^{0,1}$.
 \item $h(z)$ acts by multiplication by $z^{-1}$ (resp. $\overline{z}^{-1}$) on $V^{1,0}$ (resp. $V^{0,1}$). 

\end{enumerate}

Such a pair $(V, h)$ is called a $\QQ$-rational Hodge structure of weight $1$. If we restrict $h$ to the set of norm $1$ complex numbers (denoted $\mathbb{U}$), then the Mumford-Tate group $MT(V)$ of $(V, h)$ is defined to be the smallest algebraic subgroup of $\GL(V)$ (over $\QQ$) such that $h|_{\mathbb{U}}: \mathbb{U} \rightarrow \GL(V \otimes \RR)$ factors through $MT(V) \otimes {\RR}$. This yields a rational representation that will be denoted as $\rho: \MT(V) \hookrightarrow \GL(V)$. We will use the term Mumford-Tate representation for the pair $(G, \rho)$.

\begin{remark}
This definition only works for Hodge structures of weight $1$. Moreover, here we define what is usually called \emph{special Mumford-Tate group}, as opposed to general Mumford-Tate group in \cite{Del82}.
\end{remark}

Moreover, given a weight $1$ Mumford-Tate representation $(G, \rho)$, we can recover $V$ by looking at the target space of $\rho$; similarly we can recover $h$ by restricting $\rho$ to the maximal compact torus of $G_{\RR}$. Since $h$ decomposes $V_{\CC}$ into eigenspaces of weight $(1,0)$ and $(0,1)$, this gives us a way of assigning Hodge numbers to the tensor construction of $V_{\CC}$, similarly for $G$-invariant subspaces of such constructions. In fact, we can take $G$-invariant subspaces in $V^{\otimes n}$ to be the definition of $\QQ$-Hodge substructures in $V^{\otimes n}$.

By general Tannakian formalism (Corollary 4.5 in \cite{Moo04}), we have an equivalence of categories between \textbf{Rep}$_{\QQ}(MT(V))$ and all $\QQ$-Hodge substructures obtained by tensor operations on $V$. 

From this the following lemma is immediate:

\begin{lemma}
Given a reductive $\QQ$ group $G$ and two representations $\rho_1, \rho_2$ into $GL(V)$, if there exists a $\rho_i$-invariant subspace $W$ ($i = 1, 2$) in $V^{\otimes n}$ such that its Hodge numbers are different under $\rho_1$ and $\rho_2$, then the two representations cannot be equivalent.
\end{lemma}

Another classical theorem states that abelian varieties are determined up to isogeny by their weight $1$ Hodge structure. Therefore, in this paper we shall only consider the case when $(V,h)$ is Hodge structure of weight 1. Tannakian formalism states that this is equivalent to studying $(G, \rho)$ that can be realized as a Mumford-Tate representation of weight 1.

It is a general fact that for any representation $\rho$ of a reductive group we can associate it with a dual representation $\rho^{\vee}$. If we consider the maximal torus in $G$ containing the image of $h$, the cocharacter of $\rho^{\vee}$ with respect to that torus equals to the negative cocharacter of $\rho$. Since taking complex conjugation on $\mathbb{U}$ is the same thing as taking inverses, we have the following lemma:

\begin{lemma}
If $A$ is an abelian variety given by the Mumford-Tate representation $(G, \rho)$ up to isogeny, then the dual abelian variety $A^{\vee}$ of $A$ is given by $(G, \rho^{\vee})$.
\end{lemma}

\begin{remark}
The Mumford-Tate group of a weight 1 polarized Hodge structure always lies inside a symplectic group, whose representations are always self-dual. Hence we recover the classical fact that the dual abelian variety is isogenous to the origial abelian variety.
\end{remark}

We also need another general fact for Mumford-Tate groups:

\begin{proposition}
(Properties 2.1.4 in \cite{Moo17})Let $H_1$, $H_2$ be two $\QQ$-Hodge strucutres with Mumford-Tate groups $G_1, G_2$. Then the Mumford-Tate group of $H_1 \oplus H_2$ is an algebraic subgroup of $G_1 \times G_2$ and it admits surjective projections onto $G_1$ and $G_2$ respectively. \end{proposition}

Combining the proposition with the previous lemma we have the following:

\begin{lemma}
If $A$ is an abelian variety given by the Mumford-Tate representation $(G, \rho)$, then $A \times A^{\vee}$ is given by the Mumford-Tate representation $(G, \rho \oplus \rho^{\vee})$.
\end{lemma}

\begin{proof} From the above proposition we know the Mumford-Tate group of $A \times A^{\vee} $ admits a surjection onto $G$. We claim $G$ itself is the group, for the representation $\rho \oplus \rho^{\vee} $ will have the correct embedding of  $\mathbb{U}$ because of the definition of dual representation. 
\end{proof}

\subsection{The case of CM abelian varieties, abelian varieties of Weil type}

In this section we recall Deligne's construction of CM abelian varieties for which the Mumford-Tate group is contained in an algebraic torus (see \cite{Del82} Example 3.7).  

By a CM-field we mean a quadratic totally imaginary extension of a totally real field; a CM-algebra is a finite product of CM-fields. Let $E = E_1 \times ... \times E_n$ be such an algebra. Then there exists an involution $\iota$ acting by complex conjugation on each of the factors of $E$. Let $F = F_1 \times ... \times F_n$ denote the subalgebra in $E$ fixed by $\iota$.

We denote $S$ for the set $\Hom_{\QQ}(E, \CC)$. Then a CM-type for $E$ is a subset $\Sigma \subset S$ such that $S = \Sigma \sqcup \iota\Sigma$. 

The complex structure is given by the following decomposition:
$$E \otimes_{\QQ} \CC \cong \CC^S = \CC^{\Sigma} \oplus \CC^{\iota\Sigma}.$$

Let $\CC^{\Sigma}$ be the $V^{1,0}$ space and $\CC^{\iota\Sigma}$ its complex conjugate. Then we can view $E$ as $H_1(A, \QQ)$ of some CM abelian variety $A$, for example, given by $A(\CC) = \CC^{\Sigma} / \Sigma(\mathcal{O}_E)$. The main theorem of complex multiplication states that simple CM abelian varieties up to isogeny are classified by all possible choices $(E, \Sigma)$ up to the Galois group action on the embeddings (cf. \cite{Mil06}). For a chosen $(E, \Sigma)$ that gives $A$ up to isogeny, the polarization on $H^1(A, \QQ)$ is given by $\psi: (x, y) \mapsto \Tr_{E/\QQ}(tx\overline{y})$, where $t \in E^{\times}$ satisfies $\Im(\sigma(t)) >0 $ for all $\sigma \in \Sigma$, and $t$ is totally imaginary.

\begin{lemma} The Mumford-Tate group $G$ of $A$ given by the CM-type $(E, \Sigma)$ lies inside $E^{\times}$. In particular, it is a torus that lies inside the $\QQ$ group $ U_E = \{ x \in E^{\times} |  \Nm_{E/F} (x) = 1 \}$.
\end{lemma}

Note that, as we shall see in a minute, this does not imply $G$ is the whole $U_E$. However, the proposition from the previous section says that $G$ admits surjection onto each $U_{E_i}$. 

A general fact about $U(1)^n$ states that any map $U(1)^n \rightarrow U(1)$ is of the form $(g_1, ... g_n) \mapsto g_1^{i_1} g_2^{i_2} ... g_n^{i_n} $. We will record this information by a vector $(i_1, i_2, ..., i_n)$. Similarly, for a map

$$\begin{array}{cccc}
\rho: & U(1)^n & \rightarrow & U(1)^m \\
 & (g_1, ... ,g_n) & \mapsto & (g_1^{i_{1,1}}...g_n^{i_{1,n}}, ..., g_1^{i_{m,1}}...g_n^{i_{m,n}})
\end{array}$$
We write $$\rho = \left ( \begin{array}{ccc}
i_{1,1} & ... & i_{1,n} \\
... & ... & ... \\
i_{m,1} & ... & i_{m,n}
\end{array} \right )
$$

In particular, if $U(1)^n$ is the real form of the Mumford-Tate group of a CM abelian variety, it is known that $U(1)^n$ always embeds into $\Sp_{2n}(\RR)$ by the following identification of the following weight vectors: $$\begin{array}{ccccc}
U(1)^n\otimes \CC & \hookrightarrow & (\SU(2) \otimes \CC)^n & \hookrightarrow & \Sp_{2n}(\CC) \\
& & \CC^2 \oplus \CC^2 \oplus ... \oplus \CC^2 (\text{n copies}) & \rightarrow & \CC^{2n} \\
& & (L_1, L_2, ... , L_n) &\rightarrow & (\pm L_1, \pm L_2, ... \pm L_n) \\  
\end{array}$$
Where the choice of $\pm$ sign in the $\Sp_{2n}$ representation records the choices we made within a pair of complex embeddings of $E$. We record this information by the group homomorphism that maps $U(1)^n$ to the maximal torus in $\Sp_{2n}$ with respet to which we assign the weights $L_i$.

We call a CM abelian fourfold $A$ of Weil type if for a chosen embedding of an imaginary quadratic field $\sigma: \QQ(\sqrt{-D}) \hookrightarrow \CC$ where $\QQ(\sqrt{-D})$ is a subfield of $End^0(A)$, the action of $t \in \QQ(\sqrt{-D})$ is of the form $(\sigma(t), \sigma(t), -\sigma(t), -\sigma(t))$ on $H^{1,0}(A)$. Moreover, the count of negative numbers (up to 2, since we can always apply conjugation to get the count of signs on dual abelian varieties) in this case will be an invariant that identifies the CM type of the underlying abelian variety.

\begin{remark}\label{num}In classical literature (e.g. \cite{Zar99}), this information is recorded by a pair of unordered numbers $\{ n_{\sigma}, n_{\sigma'} \}$ where $\sigma$ denotes the chosen embedding $\QQ(\sqrt{-D}) \hookrightarrow \CC$. Naturally, in our situation, $n_{\sigma} + n_{\sigma'} = 4$. 

Moreover, a CM abelian fourfold is of Weil type if and only if $\{ n_{\sigma}, n_{\sigma'} \} = (2,2)$
\end{remark}

\section{Describing Shioda and Mumford fourfolds via Mumford-Tate representations}

\subsection{Shioda's fourfold}

It is established in Shioda's paper \cite{Shi81} (Example 6.1) that the Jacobian $A_S$ of the hyperelliptic curve $C_9: y^2 = x^9 -1$ is not simple, namely $A_S$ is a product of an abelian threefold with CM field $\mathbb{Q}(\zeta_9)$ and an elliptic curve with CM field $\mathbb{Q}(\zeta_3)$. 

By the previous chapter, we can describe this abelian fourfold via the Artinian ring $E := \QQ[x]/(x^8 + x^7 + ... +x +1) \cong \mathbb{Q}(\zeta_9) \times \mathbb{Q}(\zeta_3) = H_1(A, \mathbb{Q})$; $H_1(A, \mathbb{Z})$ can be obtained by considering the embedding of products of fractional ideals in $\mathbb{Z}(\zeta_9)$ and $\mathbb{Z}(\zeta_3)$ into $H_1(A, \mathbb{Q})$ up to a positivity condition which is induced by the Riemann condition. Therefore $A_S$ is isogeneous to the abelian variety obtained by $E/ \mathcal{O}_E$ with the complex structure given by $ E_{\CC} = \CC^{\Sigma} \oplus \CC^{\overline{\Sigma}} \oplus \CC^{\tau} \oplus \CC^{\overline{\tau}} $
where $\Sigma$ denotes the set of embeddings $\QQ(\zeta_9) \hookrightarrow \CC$ such that their restriction onto $\QQ(\zeta_3)$ is identical.

\begin{lemma}
The Mumford-Tate group of $A_S$ is $$G = \{ x \in \mathbb{Q}(\zeta_9) | Nm_{\mathbb{Q}(\zeta_9)/\mathbb{Q}(\zeta_9 + \overline{\zeta_9})}(x) = 1 \}$$

In particular, $G_{\mathbb{R}} \cong U(1)^3$ and the Galois group acts by cyclic group $A_3$ amongst the factors.
\end{lemma}

\begin{proof} From Deligne's construction we see $G$ lies inside a rank $4$ torus, and it admits surjection onto the Mumford-Tate group of the abelian threefold and the elliptic curve, which implies its rank is at least $3$. To see that it's actually rank $3$, we calculate the number of trivial representations in $H^4$ for $U(1)^4$. First we embed $U(1)^4$ into $SU(2)^4$. Then 
$$H^1 = W_{1,0,0,0} \oplus W_{0,1,0,0} \oplus W_{0,0,1,0} \oplus W_{0,0,0,1} $$ as an $SU(2)^4$ representation. One can compute
$$H^4 = \CC^6 \oplus W_{1,0,0,0} \oplus W_{0,1,0,0} \oplus ...$$
This implies that the number of (2,2)-Hodge classes would be 6, as opposed to Shioda's result, which is 8 (Example 6.1 in \cite{Shi81}). Thus completes the proof. 
\end{proof}

\begin{proposition}
(Sutherland, Zywina et al.) The representation of $G_\mathbb{R}$ on $E_\mathbb{R}$ is given by 
\begin{center}
\begin{tabular}{ c c c }
 $G_\mathbb{R} \cong U(1)^3 \hookrightarrow$ & $\SU(2)^3 \hookrightarrow$ & $\SU(2)^4$ \\ 
   & $(u,v,w) \mapsto$ & $(u,v,w,uvw)$ \\  

\end{tabular}
\end{center}

Using our notation,
$$\rho_{S, \RR} = \left ( \begin{array}{ccc}
1 & 0 & 0 \\
0 & 1 & 0 \\
0 & 0 & 1 \\
1 & 1 & 1 \\
\end{array} \right )
$$
\end{proposition}

\begin{proof} The general theory of Mumford-Tate groups states that $E_\mathbb{R} \cong \mathbb{R}^8$ splits into two irreducible, Galois invariant subrepresentations of $SU(2)^3$, one is $6$ dimensional and the other is $2$ dimensional. By general representation theory, we know that these representations must be given by direct sums and exterior tensors of representations of $SU(2)$ of highest weight $1$ (because we are working with weight $1$ space, and $SL_2(\mathbb{R})$ representations are all self-dual). Moreover, the list of weight vectors in each sub-representation should be invariant under the Galois group action, in other words, permuting the factors of exterior tensors.

It remains to show that $\Sigma$ is the set we described. In \cite{Shi81} Shioda has already fixed a canonical basis for $H^{1,0}$, namely the holomorphic 1-forms given by $\eta_v = x^{v-1}dx/y, v = 1,2,3,4$. If we perform CM by $\zeta_9$ on $y^2 = x^9 - 1$, we see that each $\eta_v$ is an eigenvector with eigenvalue $\zeta_9^v$. This implies that $\Sigma \cup \{\tau\} = \{ \zeta_9 \mapsto \{\zeta_9, \zeta_9^2, \zeta_3, \zeta_9^4\}\}$, therefore determining the representation.
\end{proof}

\begin{corollary} For Shioda's fourfold, $\{ n_{\sigma}, n_{\sigma'} \} = \{ 1,3 \}$ (see remark~\ref{num} for the definition of $\{n_{\sigma}, n_{\sigma'} \}$). In this case, the element $t$ inducing polarization can be taken as $t = \sqrt{-3}$.
\end{corollary}

\subsection{A general characterization of Mumford fourfold with CM structure}

A description of the construction of general Mumford fourfolds can be found in \cite{Mum69}, which gives rise to an isogeny class of abelian fourfolds. We will describe the Mumford-Tate representation $(G, \rho)$ when such a fourfold comes with a CM structure.

To begin with, the map $h: \mathbb{U} \rightarrow GL(\RR^8)$ is the same as the general case:

$$h(e^{i \theta}) = Id_2 \otimes Id_2 \otimes \left ( \begin{array}{cc}
cos \theta & -sin \theta \\
sin \theta & cos \theta \\

\end{array}
\right )
$$

The Mumford-Tate group $G$ in this case will be a torus defined over a totally real cubic field $K$ with $\QQ$-dimension 3. In other words, $G \otimes_{\QQ} \RR \cong U(1) \times U(1) \times U(1)$, each factor is given by the different embeddings $\sigma_i: K \hookrightarrow \RR$ where $i = 1,2,3$.

The representation $\rho_{\CC}$ is given by tensoring three copies of $2$-dimensional representations of $U(1)$, the standard notation of which is $W_{1,1,1}$. If we write out the basis of $W_{1,1,1}$ by $\{v_{i,j,k} | i, j, k = \pm 1 \}$ with each subindex remembering the weight, then we can also write down the symplectic form inducing the polarization
$$\omega =  v_{-1,-1,1}\wedge v_{1,1,-1} - v_{1,-1,1}\wedge v_{-1,1,-1} - v_{-1,1,1}\wedge v_{1,-1,-1} + v_{1,1,1}\wedge v_{-1,-1,-1} $$

The form is written explicitly in a way such that it is a sum of $(1,0)$ form wedged with $(0,1)$ form. Since $G$ must preseve this form, and by its construction it cannot take one summand to another, we have another way of writing down $\rho_{M, \RR}$, namely

$$
\begin{array}{cccc}

\rho_{M, \RR}: & U(1)^3 & \rightarrow & U(1)^4 \\
 & (e^{i \theta_1}, e^{i \theta_2}, e^{i \theta_3}) & \mapsto & (e^{i (\theta_1 + \theta_2 + \theta_3)}, e^{i (-\theta_1 - \theta_2 + \theta_3)}, e^{i ( \theta_1 - \theta_2 + \theta_3)}, e^{i ( - \theta_1 + \theta_2 + \theta_3)}) \\

\end{array}
$$

The Galois group of $K$ will shuffle the $\theta_i$'s.

\begin{lemma}
Suppose $K$ is Galois over $\QQ$ and its Galois group is given by $\{1, \sigma, \sigma^2\}$ such that for $(g_1, g_2, g_3) \in G_{\RR}$, $\sigma (g_1, g_2, g_3) = (g_2, g_3, g_1) $. If $g$ is an element in $G$ and we denote $g'$ to be $g^{-1}  g^{\sigma} g^{\sigma^2}$, then we can write $\rho_M$ into

$$\rho_{M, \RR} = \left ( \begin{array}{ccc}
-1 & 0 & 0 \\
0 & 1 & 0 \\
0 & 0 & 1 \\
1 & 1 & 1 \\
\end{array} \right )
$$
\end{lemma}

It is known that CM Mumford fourfolds are of Weil type (see 3.6, 3.7 of \cite{Noo01} and \cite{Pol68}). We observe that in the case when $B \otimes_{K} K(\sqrt{-D}) = M_2(K(\sqrt{-D}))$ for some $D \in \ZZ$, the norm-1 group of $B$ contains $\{x \in K(\sqrt{-D}) | \Norm_{K(\sqrt{-D})/K}(x) = 1 \}$ (see \cite{Voi17} Lemma 5.4.7 for the proof). In this case, the Hodge structure of the resulting CM Mumford fourfold splits into a CM threefold and a CM elliptic curve, and the information of $\sqrt{-D}$ can be in fact recovered by the CM condition on the elliptic curve. Since in any case the unordered pair $n_{\sigma} = n_{\sigma'} = 2$ does not equal to $\{1,3\}$, we have indeed proved the following:

\begin{theorem} $A_S$ is not a CM special case of Mumford's constructions.
\end{theorem}

\begin{remark} We first noticed that $A_S$ cannot be a special CM Mumford fourfold by calculating the Hodge numbers of its Hodge substructures in $Sym^2(H^1(A_S, \QQ))$. It is known that in this case, any Mumford fourfold would give rise to a Hodge substructure of K3 type with Hodge number $(1,4,1)$ (cf. \cite{Gal00} for the calculation of generic case). However, the only 6-dimensional Hodge substructure in $Sym^2(H^1(A_S, \QQ))$ has Hodge number $(3,0,3)$.
\end{remark}

\section{Obtaining $A_M$ from $A_S$}

As the previous proposition has noted, if we wish to build a CM Mumford fourfold $A_M$ from Shioda fourfold $A_S$, then the CM field of $A_M$ should be given by the CM field of the elliptic curve component of $A_S$, namely $\QQ(\sqrt{-3})$. 

To find a basis for the period matrix of $A_M$, we recall the fact that $\rho_M$ can be obtained from $\rho_S$ by flipping one of the embeddings $\QQ(\zeta_9) \hookrightarrow \CC$ to its complex conjugate. The main theorem of CM abelian varieties then states that such a flipping is possible by twisting the trace pairing on the threefold part of $A_S$ by a non-totally-positive element in $K = \QQ(\zeta_9 + \overline{\zeta_9})$. One such element is readily available, namely the primitive unit of $\mathcal{O}_K$. Therefore we obtain a basis for calculating the period matrix of a CM Mumford fourfold: 

\begin{proposition} Let $\eta_v = x^{v-1}dx/y, v = 1,2,3,4 $ be the canonical holomorphic 1-form of the Jacobian $A_S$ of the curve $y^2 = x^9-1$ equipped with intersection form $<-,->_S$ by integration. Then the period matrix of a CM Mumford abelian variety $A_M$ can be obtained by setting the basis $\omega_1, ..., \omega_8$, with $<\omega_i, \omega_j>_M = (\zeta_9^i + \zeta_9^{-i})<\eta_i, \eta_j>_S, i=1,2,3,4$.
\end{proposition}






\begin{thebibliography}{}

\bibitem{Pol68}
Pohlmann, Henry. \emph{Algebraic cycles on abelian varieties of complex multiplication type} Annals of Mathematics (1968): 161-180.

\bibitem{Mum69}
  Mumford, David,
  \emph{A Note of Shimura's Paper "Discontinuous Groups and Abelian Varieties"},
  Math. Ann. 181, 345 - 351,
  1969

\bibitem{Shi81}
Shioda, Tetsuji. \emph{Algebraic cycles on abelian varieties of Fermat type.} Mathematische Annalen 258.1, 65-80,
1981

\bibitem{Del82}
Deligne, Pierre. \emph{Hodge cycles on abelian varieties} Hodge cycles, motives, and Shimura varieties. Springer, Berlin, Heidelberg, 1982. 9-100.

\bibitem{Zar99}
Zarhinº, Y. G., \& Moonen, B. J. J. (1999). \emph{Weil classes and Rosati involutions on complex abelian} In Recent Progress in Algebra: An International Conference on Recent Progress in Algebra, August 11-15, 1997, KAIST, Taejon, South Korea (Vol. 224, p. 229). American Mathematical Soc..

\bibitem{EMO00}
Edixhoven, Sebastiaan., Moonen, Ben., Oort, Frank.  
\emph{Open problems in algebraic geometry}, 
2000.


\bibitem{Gal00}
  Galluzzi, Federica.
  \emph{Abelian Fourfold of Mumford-type and Kuga-Satake Varieties},
  Indag. Mathem., N.S., 11 (4), 547-560, 
  December 18, 2000

\bibitem{Noo00}
 Noot, Rutger. \emph{Abelian varieties with l-adic Galois representation of Mumford's type}, 
Journal fur die Reine und Angewandte Mathematik 519 (2000): 155-170.

\bibitem{Noo01}
Noot, Rutger. \emph{On Mumford's families of abelian varieties} Journal of Pure and Applied Algebra 157, no. 1 (2001): 87-106.

\bibitem{Orl02}
Orlov, Dmitri Olegovich. \emph{Derived categories of coherent sheaves on abelian varieties and equivalences between them}
 Izvestiya: Mathematics 66, no. 3 (2002): 569.

\bibitem{Moo04}
Moonen, Ben. \emph{An introduction to Mumford-Tate groups}
unpublished lecture notes, Amsterdam,
2004

\bibitem{Mil06} Milne, J.S., 2006. 
\emph{Complex multiplication.} Available at the author’s webpage.


\bibitem{Moo17}
Moonen, Ben. \emph{Families of Motives and the Mumford-Tate Conjecture}
Milan Journal of Mathematics 85(2), 257-307,
2017

\bibitem{Voi17}
  Voight, John, 
  \emph{Quaternion Algebras},
  Preprint, accessable via \url{https://math.dartmouth.edu/~jvoight/quat.html} ,
  2017


\end{thebibliography}
\end{document}